\documentclass[12pt]{amsart}
\usepackage{fullpage}

\usepackage{setspace}
\RequirePackage{amssymb}
\usepackage{amsbsy}
\usepackage{enumerate}
\usepackage{mathrsfs}
\usepackage{undertilde}
\usepackage{amsmath}
\usepackage{url}

\theoremstyle{plain}
\newtheorem{thm}{Theorem}[section]
\newtheorem*{thm*}{Theorem}
\newtheorem*{mainthm*}{Main Theorem}
\newtheorem{lem}[thm]{Lemma} \newtheorem*{lem*}{Lemma}
 \newtheorem*{claim*}{Claim}
\newtheorem{cor}[thm]{Corollary} \newtheorem*{cor*}{Corollary}
 \newtheorem*{prop*}{Proposition}

\theoremstyle{definition}
\newtheorem{defn}[thm]{Definition} \newtheorem*{defn*}{Definition}

\theoremstyle{remark}
\newtheorem{rem}[thm]{Remark} \newtheorem*{rem*}{Remark}
 \newtheorem*{example*}{Example}
\newtheorem{question}[thm]{Question} \newtheorem*{question*}{Question}

\newcommand{\Ord}{\mathrm{Ord}}

\newcommand{\bfPi}{\utilde{\bf{\Pi}}}

\DeclareMathOperator{\Con}{Con}
\DeclareMathOperator{\Col}{Col}
\DeclareMathOperator{\crit}{crit}

\DeclareMathOperator{\ran}{range}
\DeclareMathOperator{\rank}{rank}

\begin{document}

\author{Trevor M.\ Wilson}
\title{Weakly remarkable cardinals, Erd\H{o}s cardinals, and the generic Vop\v{e}nka principle}

\address{Department of Mathematics\\Miami University\\Oxford, Ohio 45056\\USA}
\email{twilson@miamioh.edu} 
\urladdr{https://www.users.miamioh.edu/wilso240}

\begin{abstract}
 We consider a weak version of Schindler's remarkable cardinals that may fail to be $\Sigma_2$-reflecting. We show that the $\Sigma_2$-reflecting weakly remarkable cardinals are exactly the remarkable cardinals, and we show that the existence of a non-$\Sigma_2$-reflecting weakly remarkable cardinal has higher consistency strength: it is equiconsistent with the existence of an $\omega$-Erd\H{o}s cardinal. We give an application involving gVP, the generic Vop\v{e}nka principle defined by Bagaria, Gitman, and Schindler. Namely, we show that gVP + ``Ord is not $\Delta_2$-Mahlo'' and $\text{gVP}({\bfPi}_1)$ + ``there is no proper class of remarkable cardinals'' are both equiconsistent with the existence of a proper class of $\omega$-Erd\H{o}s cardinals, extending results of Bagaria, Gitman, Hamkins, and Schindler.
\end{abstract}

\maketitle

\section{Remarkability and weak remarkability}

Many large cardinal properties can be defined in terms of elementary embeddings between set-sized structures. For example, extendibility is defined in terms of elementary embeddings between rank initial segments of $V$, and supercompactness admits a similar characterization by Magidor \cite[Theorem 1]{MagCombinatorialCharacterization}. Any large cardinal property defined in this way can be \emph{virtualized} by weakening the existence of an elementary embedding to the existence of a \emph{generic elementary embedding}, meaning an elementary embedding that exists in some generic extension of $V$ (and whose domain and codomain are in $V$.) The large cardinal properties obtained in this way are known as \emph{virtual large cardinal properties} (see Gitman and Schindler \cite{GitSchVirtualLargeCardinals}.) The first virtual large cardinals to be studied were the virtually supercompact cardinals, also known as the remarkable cardinals:

\begin{defn}[Schindler\footnote{Schindler \cite{SchProperRemarkable} originally gave another definition that did not involve forcing but was otherwise more complicated. See Bagaria, Gitman, and Schindler \cite[Proposition 2.4]{BagGitSchGenericVopenka} for several equivalent forms of remarkability.}]
 A cardinal $\kappa$ is \emph{remarkable} if for every ordinal $\lambda > \kappa$ there is an ordinal $\bar{\lambda}<\kappa$ and a generic elementary embedding $j:V_{\bar{\lambda}} \to V_\lambda$ such that $j(\crit(j)) = \kappa$.
\end{defn}

We will consider a weak form of remarkability obtained by removing the condition $\bar{\lambda}<\kappa$,
analogous to the weak form of virtual extendibility defined by Gitman and Hamkins \cite[Definition 6]{GitHamOrdNotDelta2Mahlo}.  We work in ZFC unless otherwise stated.

\begin{defn}
 A cardinal $\kappa$ is \emph{weakly remarkable} if for every ordinal $\lambda > \kappa$ there is an ordinal $\bar{\lambda}$ and a generic elementary embedding $j:V_{\bar{\lambda}} \to V_\lambda$ such that $j(\crit(j)) = \kappa$.
\end{defn}

In terms of consistency strength, remarkable cardinals and weakly remarkable cardinals are between ineffable cardinals and $\omega$-Erd\H{o}s cardinals. If there is an $\omega$-Erd\H{o}s cardinal then there is a transitive set model of ZFC + ``there is a remarkable cardinal'' by Schindler \cite[Lemma 1.2]{SchProperForcingRemarkableII}, and if $\kappa$ is weakly remarkable then by taking $\lambda = \kappa + 1$ in the definition one can easily show that $\kappa$ is ineffable and $V_\kappa$ satisfies ZFC + ``$\crit(j)$ is ineffable.''

The consistency strength of remarkable cardinals and weakly remarkable cardinals can be described more precisely in terms of the hierarchy of $\alpha$-iterable cardinals defined by Gitman \cite{GitRamsey}: they are between 1-iterable cardinals and 2-iterable cardinals. See Gitman and Welch \cite{GitWelRamseyII} for more information on $\alpha$-iterable cardinals.

A cardinal $\kappa$ is called \emph{$\Sigma_n$-reflecting} if it is inaccessible and $V_\kappa \prec_{\Sigma_n} V$. This definition is particularly natural in the case $n = 2$: the $\Sigma_2$ statements about a parameter $x$ are the statements that can be expressed in the form ``there is an ordinal $\lambda$ such that $V_\lambda \models \varphi[x]$'' where $\varphi$ is a formula in the language of set theory, so a cardinal $\kappa$ is $\Sigma_2$-reflecting if and only if it is inaccessible and for every formula $\varphi$ in the language of set theory, every ordinal $\lambda$, and every set $x \in V_\kappa$, if $V_\lambda \models \varphi[x]$ then $V_{\bar{\lambda}} \models \varphi[x]$ for some ordinal $\bar{\lambda} < \kappa$.

If $\kappa$ is a remarkable cardinal then for every ordinal $\lambda > \kappa$ and every set $x \in V_\lambda$ there is an ordinal $\bar{\lambda}<\kappa$ and a generic elementary embedding $j:V_{\bar{\lambda}} \to V_\lambda$ such that $j(\crit(j)) = \kappa$ and having the additional property that $x \in \ran(j)$: see Bagaria, Gitman, and Schindler \cite[Propositions 2.3 and 3.2]{BagGitSchGenericVopenka}. The same argument establishes the corresponding fact without the condition $\bar{\lambda}<\kappa$ for weakly remarkable cardinals. Note that in the case $x \in V_\kappa$, every generic elementary embedding $j:V_{\bar{\lambda}} \to V_\lambda$ such that $j(\crit(j)) = \kappa$ and $x \in \ran(j)$ must fix $x$. This implies that every remarkable cardinal is $\Sigma_2$-reflecting, but because the definition of weak remarkability lacks the condition $\bar{\lambda}<\kappa$ we cannot similarly conclude that every weakly remarkable cardinal is $\Sigma_2$-reflecting.

The following result, proved in Section \ref{sec:proof-of-main-thm}, says that the $\Sigma_2$-reflecting weakly remarkable cardinals are precisely the remarkable cardinals:

\begin{thm}\label{thm:weakly-remarkable-and-sigma-2-reflecting}
 For every cardinal $\kappa$, the following statements are equivalent.
 \begin{samepage}
 \begin{enumerate}
  \item $\kappa$ is remarkable.
  \item $\kappa$ is weakly remarkable and $\Sigma_2$-reflecting.
 \end{enumerate}
 \end{samepage}
\end{thm}

By contrast, the existence of a non-$\Sigma_2$-reflecting weakly remarkable cardinal has higher consistency strength than the existence of a remarkable cardinal: we will show that it is equiconsistent with the existence of an $\omega$-Erd\H{o}s cardinal.  (This is an unusual situation. More typically for a large cardinal property X either ZFC proves that every X cardinal is $\Sigma_2$-reflecting or ZFC proves that the least X cardinal is not $\Sigma_2$-reflecting.)
 
Following Baumgartner \cite{BauIneffabilityII}, we say that an infinite cardinal $\eta$ is \emph{$\omega$-Erd\H{o}s} if for every club $C$ in $\eta$ and every function $f : [C]^{\mathord{<}\omega} \to \eta$ that is regressive, meaning that $f(a) < \min(a)$ for all $a$ in the domain of $f$, there is a subset $X \subset C$ of order type $\omega$ that is homogeneous for $f$, meaning that $f \restriction [X]^n$ is constant for all $n < \omega$. Schmerl \cite[Theorem 6.1]{SchKappaLike} showed that the least cardinal $\eta$ such that $\eta \to (\omega)^{\mathord{<}\omega}_2$ has this property, if it exists.
 
We will not directly use the definition of $\omega$-Erd\H{o}s cardinals in terms of club sets and regressive functions, only the following consequences of the definition. First, every $\omega$-Erd\H{o}s cardinal is inaccessible. Second, if $\eta$ is an $\omega$-Erd\H{o}s cardinal then $\eta \to (\omega)^{\mathord{<}\omega}_\alpha$ for every cardinal $\alpha < \eta$. Third, if $\alpha \ge 2$ is a cardinal and there is a cardinal $\eta$ such that $\eta \to (\omega)^{\mathord{<}\omega}_\alpha$, then the least such cardinal $\eta$ is an $\omega$-Erd\H{o}s cardinal (and is greater than $\alpha$.) It follows that the statements ``there is an $\omega$-Erd\H{o}s cardinal'' and ``there is a proper class of $\omega$-Erd\H{o}s cardinals'' are equivalent to (and are convenient abbreviations of) the statements $\exists \eta\, \eta \to (\omega)^{\mathord{<}\omega}_2$ and $\forall \alpha\,\exists \eta\,\eta \to (\omega)^{\mathord{<}\omega}_\alpha$ respectively.
 
The following two results describe the relationship between $\omega$-Erd\H{o}s cardinals and non-$\Sigma_2$-reflecting weakly remarkable cardinals. They will also be proved in Section \ref{sec:proof-of-main-thm}.

\begin{thm}\label{thm:omega-erdos-is-limit-of-non-refl-weakly-rmk} 
  Every $\omega$-Erd\H{o}s cardinal is a limit of non-$\Sigma_2$-reflecting weakly remarkable cardinals.
\end{thm}

\begin{thm}\label{thm:non-refl-weakly-rmkp-implies-omega-erdos-in-L}
  If $\kappa$ is a non-$\Sigma_2$-reflecting weakly remarkable cardinal, then some ordinal greater than $\kappa$ is an $\omega$-Erd\H{o}s cardinal in $L$.
\end{thm}

We obtain the following immediate consequence:

\begin{cor}\label{cor:equivalent-in-L}
 The following statements are equiconsistent modulo ZFC and are equivalent modulo ZFC + $V = L$.
 \begin{samepage}
 \begin{enumerate}
  \item There is an $\omega$-Erd\H{o}s cardinal.
  \item There is a non-$\Sigma_2$-reflecting weakly remarkable cardinal.
 \end{enumerate}
 \end{samepage}
\end{cor}

I don't know if the two statements in Corollary \ref{cor:equivalent-in-L} are equivalent in ZFC:

\begin{question}
 Does the existence of a non-$\Sigma_2$-reflecting weakly remarkable cardinal imply the existence of an $\omega$-Erd\H{o}s cardinal, provably in ZFC?
\end{question}

Because the existence of an $\omega$-Erd\H{o}s cardinal has higher consistency strength than the existence of a remarkable cardinal, it follows from Theorems \ref{thm:weakly-remarkable-and-sigma-2-reflecting} and \ref{thm:non-refl-weakly-rmkp-implies-omega-erdos-in-L} that the two theories ZFC + ``there is a weakly remarkable cardinal'' and ZFC + ``there is a remarkable cardinal'' are equiconsistent. The following result shows that they are not equivalent (assuming the existence of an $\omega$-Erd\H{o}s cardinal is consistent with ZFC.)

\begin{cor}
 The following statements are equiconsistent modulo ZFC.
 \begin{samepage}
 \begin{enumerate}
  \item\label{item:omega-erdos} There is an $\omega$-Erd\H{o}s cardinal.
  \item\label{item:weakly-rmk-no-refl} There is a weakly remarkable cardinal and there is no $\Sigma_2$-reflecting cardinal.
  \item\label{item:weakly-rmk-no-rmk} There is a weakly remarkable cardinal and there is no remarkable cardinal.
 \end{enumerate}
 \end{samepage}
\end{cor}
\begin{proof}
 $\Con(\ref{item:omega-erdos})$ implies $\Con(\ref{item:weakly-rmk-no-refl})$:
 Assume there is an $\omega$-Erd\H{o}s cardinal.
 Passing from $V$ to $V_\lambda$ where $\lambda$ is the least $\Sigma_2$-reflecting cardinal if it exists, we may assume there is no $\Sigma_2$-reflecting cardinal.
 Because the existence of an $\omega$-Erd\H{o}s cardinal is a $\Sigma_2$ statement, it is preserved by this step and the resulting model ($V$ or $V_\lambda$) has a weakly remarkable cardinal by Theorem~\ref{thm:omega-erdos-is-limit-of-non-refl-weakly-rmk}.
 
 Statement \ref{item:weakly-rmk-no-refl} implies statement \ref{item:weakly-rmk-no-rmk} because remarkable cardinals are $\Sigma_2$-reflecting.
 
 $\Con(\ref{item:weakly-rmk-no-rmk})$ implies $\Con(\ref{item:omega-erdos})$:
 If statement \ref{item:weakly-rmk-no-rmk} holds then there is a weakly remarkable cardinal that is not remarkable, and therefore is not $\Sigma_2$-reflecting by Theorem \ref{thm:weakly-remarkable-and-sigma-2-reflecting}, so there is an $\omega$-Erd\H{o}s cardinal in $L$ by Theorem \ref{thm:non-refl-weakly-rmkp-implies-omega-erdos-in-L}.
\end{proof}

In Section \ref{sec:proof-of-main-thm} we will prove Theorems \ref{thm:weakly-remarkable-and-sigma-2-reflecting},
\ref{thm:omega-erdos-is-limit-of-non-refl-weakly-rmk}, and \ref{thm:non-refl-weakly-rmkp-implies-omega-erdos-in-L}. In Section \ref{sec:generic-vopenka} we will give an application involving the generic Vop\v{e}nka principle defined by Bagaria, Gitman, and Schindler~\cite{BagGitSchGenericVopenka}.

\section{Proof of Theorems \ref{thm:weakly-remarkable-and-sigma-2-reflecting},
\ref{thm:omega-erdos-is-limit-of-non-refl-weakly-rmk}, and \ref{thm:non-refl-weakly-rmkp-implies-omega-erdos-in-L}}\label{sec:proof-of-main-thm}

We will need the following local forms of remarkabilty and weak remarkability.

\begin{defn}
 Let $\kappa$ be a cardinal and let $\lambda > \kappa$ be an ordinal.
 \begin{enumerate}
  \item $\kappa$ is \emph{$\lambda$-remarkable} if there is an ordinal $\bar{\lambda} < \kappa$ and a generic elementary embedding $j:V_{\bar{\lambda}} \to V_\lambda$ such that $j(\crit(j)) = \kappa$.\footnote{This definition is unrelated to $n$-remarkability for a positive integer $n$ as defined by Bagaria, Gitman, and Schindler \cite[Definition 3.1]{BagGitSchGenericVopenka}.}
  \item $\kappa$ is \emph{weakly $\lambda$-remarkable} if there is an ordinal $\bar{\lambda}$ and a generic elementary embedding $j:V_{\bar{\lambda}} \to V_\lambda$ such that $j(\crit(j)) = \kappa$.
 \end{enumerate}
\end{defn}

By definition, $\kappa$ is remarkable if and only if it is $\lambda$-remarkable for every ordinal $\lambda > \kappa$, and $\kappa$ is weakly remarkable if and only if it is weakly $\lambda$-remarkable for every ordinal $\lambda > \kappa$. 

\begin{defn}
 Let $\kappa$ be a cardinal and let $\lambda > \kappa$ be an ordinal.
 \begin{enumerate}
  \item $\kappa$ is \emph{$\mathord{<}\lambda$-remarkable} if it is $\beta$-remarkable for every ordinal $\beta$
  with $\kappa < \beta < \lambda$.
  \item $\kappa$ is \emph{weakly $\mathord{<}\lambda$-remarkable} if it is weakly $\beta$-remarkable for every ordinal $\beta$ with $\kappa < \beta < \lambda$.
 \end{enumerate}
\end{defn}
 
By a well-known absoluteness lemma (see Bagaria, Gitman, and Schindler \cite[Lemma 2.6]{BagGitSchGenericVopenka}) if an elementary embedding $j:V_{\bar{\lambda}} \to V_\lambda$ such that $j(\crit(j)) = \kappa$ exists in some generic extension of $V$ and $g\subset \Col(\omega,V_{\bar{\lambda}})$ is a $V$-generic filter, then some such elementary embedding exists in $V[g]$. One consequence of this fact is that $\lambda$-remarkability, $\mathord{<}\lambda$-remarkability, and their weak forms are absolute between $V$ and $V_{\lambda'}$ for every limit cardinal $\lambda' > \lambda$. A further consequence of this fact is that remarkability and weak remarkability are $\Pi_2$ properties.
 
\begin{proof}[Proof of Theorem \ref{thm:weakly-remarkable-and-sigma-2-reflecting}]
 It is clear that every remarkable cardinal is weakly remarkable and $\Sigma_2$-reflecting. Conversely, suppose that $\kappa$ is weakly remarkable and $\Sigma_2$-reflecting. We will show that $\kappa$ is $\lambda$-remarkable for every ordinal $\lambda > \kappa$ by induction on $\lambda$. Let $\lambda > \kappa$ and assume that $\kappa$ is $\mathord{<}\lambda$-remarkable.  Because $\kappa$ is weakly $(\lambda+\omega)$-remarkable there is an ordinal of the form $\bar{\lambda} + \omega$ and a generic elementary embedding
 \[j: V_{\bar{\lambda}+\omega} \to V_{\lambda+\omega} \text{ with } j(\bar{\lambda}) = \lambda \text{ and } j(\bar{\kappa}) = \kappa\]
 where $\bar{\kappa} = \crit(j)$.  If $\bar{\lambda} < \kappa$ then the restriction $j \restriction V_{\bar{\lambda}}$ witnesses that $\kappa$ is $\lambda$-remarkable and we are done.  Therefore we suppose that $\bar{\lambda} \ge \kappa$.
 
 The fact that $\kappa$ is $\mathord{<}\lambda$-remarkable is absolute to $V_{\lambda+\omega}$, so by the elementary of $j$ the model $V_{\bar{\lambda}+\omega}$ satisfies ``$\bar{\kappa}$ is $\mathord{<}\bar{\lambda}$-remarkable'' and it follows that $\bar{\kappa}$ really is $\mathord{<}\bar{\lambda}$-remarkable.  Then $\bar{\kappa}$ is $\mathord{<}\kappa$-remarkable because $\bar{\lambda} \ge \kappa$. Equivalently, $\bar{\kappa}$ is remarkable in $V_\kappa$. Because remarkability is a $\Pi_2$ property and $V_\kappa \prec_{\Sigma_2} V$, it follows that  $\bar{\kappa}$ is remarkable in $V$.  In particular, $\bar{\kappa}$ is $\bar{\lambda}$-remarkable, and this fact is absolute to $V_{\bar{\lambda}+\omega}$.  By the elementarity of $j$, the model $V_{\lambda+\omega}$ satisfies ``$\kappa$ is $\lambda$-remarkable'' and this fact is absolute to $V$.
\end{proof}
 
Next we will prove Theorem \ref{thm:omega-erdos-is-limit-of-non-refl-weakly-rmk}.  In fact we will prove a stronger result in terms of the following definition.

\begin{defn}[Gitman and Hamkins {\cite[Definition 6]{GitHamOrdNotDelta2Mahlo}}]\label{defn:wvAe}
 Let $\kappa$ be a cardinal and let $A$ be a class. Then $\kappa$ is \emph{weakly virtually $A$-extendible} if for every ordinal $\lambda > \kappa$ there is an ordinal $\theta$ and a generic elementary embedding $j: (V_\lambda; \mathord{\in}, A \cap V_\lambda) \to (V_\theta; \mathord{\in}, A \cap V_\theta)$ with $\crit(j) = \kappa$.
\end{defn}
 
This definition can be used in the context of GB + AC, meaning G\"{o}del--Bernays set theory with the axiom of choice but without the axiom of global choice. Any model of ZFC together with its definable (from parameters) classes gives a model of GB + AC, but there may be models of GB + AC with classes that are not definable.
 
A cardinal is called \emph{weakly virtually extendible} if it is weakly virtually $\emptyset$-extendible, meaning simply that for every ordinal $\lambda > \kappa$ there is an ordinal $\theta$ and a generic elementary embedding $j: V_\lambda \to V_\theta$ with $\crit(j) = \kappa$. The following lemma is similar to the fact that every extendible cardinal is supercompact, which is due to Magidor \cite[Lemma 2]{MagRoleSupercompact}.
 
\begin{lem}
 Every weakly virtually extendible cardinal is weakly remarkable.
\end{lem}
\begin{proof}
 Let $\kappa$ be a weakly virtually extendible cardinal and let $\lambda > \kappa$ be an ordinal. Then there is an ordinal $\theta$ and a generic elementary embedding 
 \[j : V_{\lambda+\omega} \to V_{\theta} \text{ with } \crit(j) = \kappa.\]
 The restriction $j \restriction V_{\lambda}$ witnesses that $j(\kappa)$ is weakly $j(\lambda)$-remarkable.
 Because the weak $j(\lambda)$-remarkability of $j(\kappa)$ is absolute to $V_\theta$, it follows by the elementarity of $j$ that $V_{\lambda+\omega}$ satisfies the statement ``$\kappa$ is weakly $\lambda$-remarkable,'' and this statement is absolute to $V$.
\end{proof}
 
Theorem \ref{thm:omega-erdos-is-limit-of-non-refl-weakly-rmk} may now be obtained as a consequence of the following result, whose full strength will not be needed until Section \ref{sec:generic-vopenka}:
 
\begin{lem}[GB + AC] \label{lem:omega-erdos-is-limit-of-non-refl-wvAe}
 Let $A$ be a class and let $\eta$ be an $\omega$-Erd\H{o}s cardinal. Then $\eta$ is a limit of non-$\Sigma_2$-reflecting weakly virtually $A$-extendible cardinals.
\end{lem}
\begin{proof}
 Let $\alpha < \eta$ be an infinite cardinal.  We will show that there is a non-$\Sigma_2$-reflecting virtually $A$-extendible cardinal between $\alpha$ and $\eta$. We may assume without loss of generality (by decreasing $\eta$ if necessary) that $\eta$ is the least $\omega$-Erd\H{o}s cardinal greater than $\alpha$. Then because the $\omega$-Erd\H{o}s property is $\Sigma_2$, there is no $\Sigma_2$-reflecting cardinal between $\alpha$ and $\eta$, so it suffices to show that there is a weakly virtually $A$-extendible cardinal between $\alpha$ and $\eta$.
  
 First, we will show that for every ordinal $\lambda \ge \eta$ there is a generic elementary embedding
 \[ j : (V_\lambda;\mathord{\in},A \cap V_\lambda) \to (V_\lambda;\mathord{\in}, A \cap V_\lambda) \text{ with }\alpha < \crit(j) < \eta.\]
 We follow the argument of Gitman and Schindler \cite[Theorem 4.17]{GitSchVirtualLargeCardinals}, who proved this in the case $\lambda = \eta$ and $A = \emptyset$. Let $\lambda \ge \eta$ and take a set $D \subset \beth_\lambda$ coding the structure $(V_\lambda; \mathord{\in},A \cap V_\lambda)$. Because $\eta\to (\omega)^{\mathord{<}\omega}_{2^\alpha}$, the structure 
 \[\mathcal{M} = (L_{\beth_\lambda^+}[D]; \mathord{\in}, D, \eta, \xi)_{\xi \le \alpha}\]
 has a set of indiscernibles $I \subset \eta$ of order type $\omega$. Let $X$ be the Skolem hull of $I$ in $\mathcal{M}$.  Note that $X$ has cardinality $\alpha$ and it contains $\eta$ and all ordinals $\xi \le \alpha$ because they are part of the language of $\mathcal{M}$. Let  $\bar{\mathcal{M}}$ be the transitive collapse of $X$ and let $\bar{\eta}$ be the image of $\eta$ under this transitive collapse. Then the uncollapse map gives an elementary embedding
 \[\pi : \bar{\mathcal{M}} \to \mathcal{M} \text{ with } \crit(\pi) > \alpha \text{ and } \pi(\bar{\eta}) = \eta.\]
 We have a generating set of indiscernibles $\pi^{-1}[I] \subset \bar{\eta}$ for $\bar{\mathcal{M}}$ of order type $\omega$, and shifting these indiscernibles by 1 gives an elementary embedding
 \[j : \bar{\mathcal{M}} \to \bar{\mathcal{M}} \text{ with } \alpha < \crit(j) < \bar{\eta}.\]
  
 Because the predicate of $\bar{\mathcal{M}}$ codes the structure $\pi^{-1}(V_\lambda; \mathord{\in},A \cap V_\lambda)$, the map $j \restriction \pi^{-1}(V_\lambda)$ is an elementary embedding from the structure $\pi^{-1}(V_\lambda; \mathord{\in},A \cap V_\lambda)$ to itself with critical point between $\alpha$ and $\bar{\eta}$. By the usual absoluteness lemma, $\bar{\mathcal{M}}$ therefore satisfies the statement ``there is a generic elementary embedding from the structure $\pi^{-1}(V_\lambda; \mathord{\in},A \cap V_\lambda)$ to itself with critical point between $\alpha$ and $\bar{\eta}$.'' By the elementarity of $\pi$, it follows that $\mathcal{M}$ satisfies the statement ``there is a generic elementary embedding from the structure $(V_\lambda; \mathord{\in},A \cap V_\lambda)$ to itself with critical point between $\alpha$ and $\eta$,'' and this statement is absolute to $V$.
 
 Now by replacement there is some cardinal $\kappa$ between $\alpha$ and $\eta$ such that for a proper class of ordinals $\lambda$ there is a generic elementary embedding 
 \[j : (V_\lambda;\mathord{\in}, A\cap V_\lambda) \to (V_\lambda; \mathord{\in}, A \cap V_\lambda) \text{ with } \crit(j) = \kappa.\]
 These generic elementary embeddings and their restrictions to the other rank initial segments of $V$ above $\kappa$ witness the weak virtual $A$-extendibility of $\kappa$.
\end{proof}

It remains to prove Theorem \ref{thm:non-refl-weakly-rmkp-implies-omega-erdos-in-L}.  First we will show that the generic elementary embeddings witnessing weak remarkability of a non-$\Sigma_2$-reflecting cardinal $\kappa$ must fix some ordinal $\beta > \kappa$:
 
\begin{lem}\label{lem:fixed-point}
 Let $\kappa$ be a non-$\Sigma_2$-reflecting weakly remarkable cardinal. Then there is an ordinal $\beta > \kappa$ such that for every ordinal $\lambda > \beta$ there is an ordinal $\bar{\lambda} > \beta$ and a generic elementary embedding $j : V_{\bar{\lambda}} \to V_\lambda$ with $j(\crit(j)) = \kappa$ and $j(\beta) = \beta$.
\end{lem}
\begin{proof}
 Because $\kappa$ is not $\Sigma_2$-reflecting, there is a formula $\varphi$ in the language of set theory, an ordinal $\beta$, and a set $x \in V_\kappa$ such that
 \begin{equation}\label{eqn:witness}
  V_\beta \models \varphi[x] \text{ and } \forall \alpha < \kappa\, V_\alpha \not\models \varphi[x].
 \end{equation}
 (Here we consider $V_\alpha \not\models \varphi[x]$ to include the case $x \notin V_\alpha$.)
 
 Fix a formula $\varphi$ such that \eqref{eqn:witness} holds for some ordinal $\beta$ and some set $x \in V_\kappa$. Define $\beta$ to be the least ordinal such that \eqref{eqn:witness} holds for some set $x \in V_\kappa$. Note that because $\kappa$ is inaccessible we have $V_\alpha \prec V_\kappa$ for a club set of $\alpha < \kappa$, so $\beta \ne \kappa$ and therefore $\beta > \kappa$. Define $\xi < \kappa$ to be the least ordinal such that \eqref{eqn:witness} holds for some set $x$ such that $\rank(x) = \xi$, and fix such a set $x$. Note that the minimality of $\beta$ implies the following strengthening of \eqref{eqn:witness}:
 \begin{equation}\label{eqn:stronger}
  V_\beta \models \varphi[x] \text{ and } \forall \alpha < \beta\, V_\alpha \not\models \varphi[x].
 \end{equation}
 Now let $\lambda > \beta$ be an ordinal. Because $\kappa$ is weakly remarkable, there is an ordinal $\bar{\lambda}$ and a generic elementary embedding
 \[j : V_{\bar{\lambda}} \to V_\lambda \text{ with } j(\bar{\kappa}) = \kappa \text{ where } \bar{\kappa} = \crit(j).\]
 The definition of $\beta$ from $\kappa$ is absolute between $V$ and $V_\lambda$, so by the elementarity of $j$ we have $\beta \in \ran(j)$, say $\beta = j(\bar{\beta})$. Note that 
 \[\bar{\kappa} < \bar{\beta} < \bar{\lambda} \text{ and } \kappa < \beta < \lambda.\]
 The definition of $\xi$ from $\beta$ and $\kappa$ is absolute between $V$ and $V_\lambda$, so by the elementarity of $j$ we have $\xi \in \ran(j)$.  Because $\xi < \kappa$ and $\kappa \cap \ran(j) = \bar{\kappa}$ we have $\xi < \bar{\kappa}$.  Therefore $j(x) = x$, so we have
 \begin{equation}\label{eqn:reflected}
   V_{\bar{\beta}} \models \varphi[x] \text{ and } \forall \alpha < \bar{\beta}\, V_\alpha \not\models \varphi[x]
 \end{equation}
 by the elementarity of $j$ and the fact that \eqref{eqn:stronger} and \eqref{eqn:reflected} are absolute to $V_\lambda$ and $V_{\bar{\lambda}}$ respectively. The conjunction of \eqref{eqn:stronger} and \eqref{eqn:reflected} implies $\bar{\beta} = \beta$, so $\bar{\lambda} > \beta$ and $j(\beta) = \beta$ as desired.
\end{proof}

\begin{rem}
 For any generic elementary embedding $j$ as in the conclusion of Lemma \ref{lem:fixed-point}, the restriction $j \restriction V_\beta$ is a generic elementary embedding from $V_\beta$ to $V_\beta$, so its critical point is by definition a \emph{virtual rank-into-rank cardinal}.  The proof of Lemma \ref{lem:fixed-point} is similar to the proof of existence of virtual rank-into-rank cardinals from a related hypothesis by Bagaria, Gitman, and Schindler \cite[Theorem 5.4]{BagGitSchGenericVopenka}.
\end{rem}

We can use generic elementary embeddings with fixed points to obtain a partition relation in $L$:
 
\begin{lem}
 Let $\kappa$ be a cardinal and let $\beta > \kappa$ be an ordinal such that for every ordinal $\lambda > \beta$ there is an ordinal $\bar{\lambda} > \beta$ and a generic elementary embedding $j : V_{\bar{\lambda}} \to V_\lambda$ such that $j(\crit(j)) = \kappa$ and $j(\beta) = \beta$. Then $\beta \to (\omega)^{\mathord{<}\omega}_\kappa$ in $L$.
\end{lem}
\begin{proof}
 Let $\lambda = \left|\beta\right|^{+\omega}$, which is more than enough for the following argument.  Take an ordinal $\bar{\lambda} > \beta$ and a generic elementary embedding $j : V_{\bar{\lambda}} \to V_\lambda$ such that, letting $\bar{\kappa} = \crit(j)$, we have $j(\bar{\kappa}) = \kappa$ and $j(\beta) = \beta$.
 For every $n<\omega$ the model $V_{\bar{\lambda}}$ thinks $\left|\beta\right|^{+n}$ exists because $j$ is elementary, and it computes cardinal successors correctly because it is a rank initial segment of $V$, so $\bar{\lambda} = \lambda$.

 Let $\gamma = (\left|\beta\right|^+)^L$ and define $\ell = j \restriction L_\gamma$, which is the only part of $j$ that we will need for the following argument.  Then $\ell$ is a generic elementary embedding and we have
 \[ \ell : L_\gamma \to L_\gamma \text{ and } \crit(\ell) = \bar{\kappa} \text{ and } \ell(\bar{\kappa}) = \kappa \text{ and } \ell(\beta) = \beta.\]
 
 Assume toward a contradiction that $\beta \not\to (\omega)^{\mathord{<}\omega}_\kappa$ in $L$. This assumption is absolute between $L$ and $L_\gamma$ because $\gamma = (\left|\beta\right|^+)^L$, so by the elementarity of $\ell^2$ and the fact that $\kappa < \ell(\kappa) = \ell^2(\bar{\kappa})$, there is some $\alpha < \bar{\kappa}$ such that $\beta \not\to (\omega)^{\mathord{<}\omega}_\alpha$ in $L_\gamma$ and therefore in $L$. Let $f : [\beta]^{\mathord{<}\omega} \to \alpha$ be the $<_L$-least witness to $\beta \not\to (\omega)^{\mathord{<}\omega}_\alpha$ in $L$ and note that this definition of $f$ is absolute between $L$ and $L_\gamma$. Then we have $\ell(f) = f$ because $\ell(\alpha) = \alpha$ and $\ell(\beta) = \beta$ and $f$ is definable from $\alpha$ and $\beta$ in $L_\gamma$. Let $(\kappa_n : n < \omega)$ be the critical sequence of $\ell$, which is defined by $\kappa_n = \ell^n(\bar{\kappa})$ for all $n < \omega$. Then by the elementarity of $\ell$ we have
 \[f(\kappa_0, \ldots, \kappa_{n-1}) = f(\kappa_1,\ldots,\kappa_n)\]
 for every positive integer $n$, so the set $\{\kappa_n : n < \omega\}$ is homogeneous for $f$ by the argument of Silver \cite[\S 2]{SilLargeCardinalInL}. The existence of a homogeneous set for $f$ of order type $\omega$ is absolute to $L$ by the argument of Silver \cite[\S 1]{SilLargeCardinalInL}, but the existence of such a homogeneous set for $f$ in $L$ contradicts our assumption that $f$ is a witness to $\beta \not\to (\omega)^{\mathord{<}\omega}_\alpha$ in $L$. 
\end{proof}

Recall that if $\beta \to (\omega)^{\mathord{<}\omega}_\kappa$ then the least ordinal $\eta$ such that $\eta \to (\omega)^{\mathord{<}\omega}_\kappa$ is an $\omega$-Erd\H{o}s cardinal greater than $\kappa$. Applying this fact in $L$ completes the proof of Theorem \ref{thm:non-refl-weakly-rmkp-implies-omega-erdos-in-L}.

\section{Application to the generic Vop\v{e}nka principle}\label{sec:generic-vopenka}

The \emph{generic Vop\v{e}nka principle}, \emph{gVP}, defined by Bagaria, Gitman, and Schindler \cite{BagGitSchGenericVopenka} says that for every proper class of structures of the same type, there is a generic elementary embedding of one of the structures into another. Gitman and Hamkins \cite[Theorem 7]{GitHamOrdNotDelta2Mahlo} proved that gVP is equivalent to the existence of a proper class of weakly virtually $A$-extendible cardinals for every class $A$ (see Definition \ref{defn:wvAe} above.) They observed that the same proof works in GBC for arbitrary classes and in ZFC for definable classes. Because the proof requires neither the axiom of global choice nor the definability of classes, it works more generally in GB + AC. Combining this result with Lemma \ref{lem:omega-erdos-is-limit-of-non-refl-wvAe}, we immediately obtain the following consequence (which is not difficult to prove directly):

\begin{lem}[GB + AC]\label{lem:pc-omega-erdos-implies-gvp}
 If there is a proper class of $\omega$-Erd\H{o}s cardinals then gVP holds.
\end{lem}

\begin{rem}
 In terms of consistency strength, gVP is weaker than the existence of a single $\omega$-Erd\H{o}s cardinal: the least $\omega$-Erd\H{o}s cardinal is a limit of virtual rank-into-rank cardinals by Gitman and Schindler \cite[Theorem 4.17]{GitSchVirtualLargeCardinals}, and if $\kappa$ is a virtual rank-into-rank cardinal then gVP holds in $V_\kappa$ with respect to its definable subsets by Bagaria, Gitman, and Schindler \cite[Proposition 3.10 and Theorem 5.6]{BagGitSchGenericVopenka}. (In fact it is not difficult to prove directly that if $\kappa$ is a virtual rank-into-rank cardinal then gVP holds in $V_\kappa$ with respect to all of its subsets.)
\end{rem}

If $n$ is a positive integer then \emph{$\text{gVP}({\bfPi}_n)$} is the fragment of the generic Vop\v{e}nka principle asserting that for every ${\bfPi}_n$-definable proper class of structures of the same type, there is a generic elementary embedding of one of the structures into another.
Arguing similarly to Gitman and Hamkins \cite[Theorem 7]{GitHamOrdNotDelta2Mahlo}, we will show that $\text{gVP}({\bfPi}_1)$ is equivalent to the existence of a proper class of weakly remarkable cardinals.

\begin{rem}
 In the non-virtual context Solovay, Reinhardt, and Kanamori \cite[Theorem 6.9]{SolReiKanStrongAxioms} proved that Vop\v{e}nka's principle is equivalent to the existence of an $A$-extendible cardinal for every class $A$, and Bagaria \cite[Corollary 4.7]{BagCnCardinals} proved that the fragment $\text{VP}({\bfPi}_1)$ of Vop\v{e}nka's principle is equivalent to the existence of a proper class of supercompact cardinals. These results use Kunen's inconsistency. In the virtual context Kunen's inconsistency is unavailable, which is why the weak forms of remarkability and virtual $A$-extendibility become relevant.
\end{rem}

\begin{lem}\label{lem:gvp-pi-1-equiv-pc-weakly-rmk}
 The following statements are equivalent.
 \begin{samepage}
 \begin{enumerate}
  \item $\text{gVP}({\bfPi}_1)$.
  \item There is a proper class of weakly remarkable cardinals.
 \end{enumerate}
 \end{samepage}
\end{lem}
\begin{proof}
 Assume $\text{gVP}({\bfPi}_1)$ and let $\alpha$ be a cardinal.  We will show there is a weakly remarkable cardinal greater than $\alpha$. Assume not, toward a contradiction. Then for every ordinal $\kappa > \alpha$ we may define $f(\kappa)$ to be the least ordinal $\lambda > \kappa$ such that $\kappa$ is not weakly $\lambda$-remarkable. (If $\kappa$ is not a cardinal, then $f(\kappa) = \kappa+1$.) For every ordinal $\beta > \alpha$, let
 \[g(\beta) = \sup\{f(\kappa) : \alpha < \kappa \le \beta\}.\]
 Consider the proper class of structures $\mathcal{C} = \{\mathcal{M}_\beta : \beta > \alpha\}$ where 
 \[\mathcal{M}_\beta = (V_{g(\beta) + \omega}; \mathord{\in}, \beta, \xi)_{\xi \le \alpha}.\] 
 The class $\mathcal{C}$ is $\Pi_1(\alpha)$, so by $\text{gVP}({\bfPi}_1)$ there are two distinct structures $\mathcal{M}_{\bar{\beta}}$ and $\mathcal{M}_\beta$ in $\mathcal{C}$ and a generic elementary embedding 
 \[j : \mathcal{M}_{\bar{\beta}} \to \mathcal{M}_\beta.\]
 We have $j(\bar{\beta}) = \beta$ and $j(\xi) = \xi$ for all $\xi \le \alpha$, so letting $\bar{\kappa} = \crit(j)$ and  $\kappa = j(\bar{\kappa})$ we have $\alpha < \bar{\kappa} \le \bar{\beta}$ and $\alpha < \kappa \le \beta$. Then we have $f(\bar{\kappa}) \le g(\bar{\beta})$ and $f(\kappa) \le g(\beta)$ by the definition of $g$ from $f$, and we have $j(f(\bar{\kappa})) = f(\kappa)$ because the definition of $f$ is absolute to $\mathcal{M}_{\bar{\beta}}$ and $\mathcal{M}_\beta$. Therefore the restriction $j \restriction V_{f(\bar{\kappa})}$ is defined and is a generic elementary embedding from $V_{f(\bar{\kappa})}$ to $V_{f(\kappa)}$ witnessing that $\kappa$ is weakly $f(\kappa)$-remarkable, contradicting the definition~of~$f$.

 Conversely, assume there is a proper class of weakly remarkable cardinals and let $\mathcal{C}$ be a ${\bfPi}_1$ proper class of structures of the same type $\tau$. Then $\mathcal{C}$ is $\Pi_1(x)$ for some set $x$.  Take a weakly remarkable cardinal $\kappa$ such that $\tau, x \in V_\kappa$. Let $F : \Ord \to \Ord$ be the strictly increasing enumeration of the class of ordinals $\{\rank(\mathcal{M}) : \mathcal{M} \in \mathcal{C}\}$ and take an ordinal $\lambda > F(\kappa)$ such that $\lambda \in C^{(1)}$, where $C^{(1)}$ denotes the class of all ordinals $\lambda$ such that $V_\lambda \prec_{\Sigma_1} V$. Bagaria \cite{BagCnCardinals} showed that $C^{(1)}$ is equal to the class of all uncountable cardinals $\lambda$ such that $V_\lambda = H_\lambda$.
 
 Because $\kappa$ is weakly remarkable, there is an ordinal $\bar{\lambda}$ and a generic elementary embedding 
 \[j : V_{\bar{\lambda}} \to V_\lambda \text{ with } j(\bar{\kappa}) = \kappa \text{ where } \bar{\kappa} = \crit(j).\]
 We may assume that $\tau$ and $x$ are in the range of $j$ because the generic embeddings witnessing weak remarkability may be taken to contain any finitely many given elements in their range (by the same proof as for remarkability, as cited in the introduction.) Because $\tau$ and $x$ are in the set $V_\kappa \cap \ran(j)$, which is equal to $V_{\bar{\kappa}}$, they are fixed by $j$.
 
 We have $V_{\bar{\lambda}} = H_{\bar{\lambda}}$ by the elementarity of $j$, so $\bar{\lambda} \in C^{(1)}$ also. Therefore the definitions of the class $\mathcal{C}$ and the class function $F$ from $x$ are absolute to $V_{\bar{\lambda}}$ as well as to $V_\lambda$, so by the elementarity of $j$ and the fact that $\lambda > F(\kappa)$ it follows that $\bar{\lambda} > F(\bar{\kappa})$. Take $\mathcal{M} \in \mathcal{C} \cap V_{\bar{\lambda}}$ with $\rank(\mathcal{M}) = F(\bar{\kappa})$. Then $j(\mathcal{M}) \in \mathcal{C} \cap V_{\lambda}$ and we have
 \[\rank(\mathcal{M}) = F(\bar{\kappa}) < F(\kappa) = \rank(j(\mathcal{M})),\] so $\mathcal{M} \ne j(\mathcal{M})$. Because the type $\tau$ of the structure $\mathcal{M}$ is fixed by $j$, the restriction $j \restriction \mathcal{M}$ is a generic elementary embedding from $\mathcal{M}$ to $j(\mathcal{M})$ as desired.
\end{proof}

We now easily obtain the following consequence, which extends results of Bagaria, Gitman, and Schindler \cite{BagGitSchGenericVopenka} as well as Gitman and Hamkins \cite{GitHamOrdNotDelta2Mahlo} (see Remark \ref{rem:previous-work} below.)

\begin{thm}\label{thm:gvp-equiconsistency}
 The following theories are equiconsistent:
 \begin{samepage}
 \begin{enumerate}
  \item\label{item:pc-omega-erdos-2} ZFC + there is a proper class of $\omega$-Erd\H{o}s cardinals.
  \item\label{item:gvp-ord-not-delta-2-mahlo} GBC + gVP + ``Ord is not $\Delta_2$-Mahlo.''
  \item\label{item:gvp-pi-1-no-pc-rmk} ZFC + $\text{gVP}({\bfPi}_1)$ + ``there is no proper class of remarkable cardinals.''
 \end{enumerate}
 \end{samepage}
\end{thm}
\begin{proof}
 $\Con(\ref{item:pc-omega-erdos-2})$ implies $\Con(\ref{item:gvp-ord-not-delta-2-mahlo})$: Assume that there is a proper class of $\omega$-Erd\H{o}s cardinals. The $\omega$-Erd\H{o}s property is downward absolute to $L$ by the argument of Silver \cite[\S 1]{SilLargeCardinalInL}, so there is a proper class of $\omega$-Erd\H{o}s cardinals in $L$. If there is an inaccessible limit of $\omega$-Erd\H{o}s cardinals in $L$, let $\lambda$ be the least such and let $M = V_\lambda^L$; otherwise let $M = L$. Because $M$ satisfies ``$V = L$'' it satisfies global choice with respect to its definable classes. Because $M$ satisfies ``there is a proper class of $\omega$-Erd\H{o}s cardinals'' it satisfies gVP with respect to its definable classes by Lemma \ref{lem:pc-omega-erdos-implies-gvp}. Finally, in $M$ the class of limits of $\omega$-Erd\H{o}s cardinals is a $\Delta_2$-definable club class of singular cardinals by our choice of $\lambda$, so $M$ satisfies ``Ord is not $\Delta_2$-Mahlo.''
 
 If theory \ref{item:gvp-ord-not-delta-2-mahlo} holds then theory \ref{item:gvp-pi-1-no-pc-rmk} holds in the first-order part of the universe because $\text{gVP}({\bfPi}_1)$ is a fragment of gVP, remarkable cardinals are $\Sigma_2$-reflecting, and the existence of a $\Sigma_2$-reflecting cardinal implies that Ord is $\Delta_2$-Mahlo.

 $\Con(\ref{item:gvp-pi-1-no-pc-rmk})$ implies $\Con(\ref{item:pc-omega-erdos-2})$: If theory \ref{item:gvp-pi-1-no-pc-rmk} holds then by Lemma \ref{lem:gvp-pi-1-equiv-pc-weakly-rmk} there is a proper class of weakly remarkable cardinals that are not remarkable, and therefore are not $\Sigma_2$-reflecting by Theorem \ref{thm:weakly-remarkable-and-sigma-2-reflecting}, so there is a proper class of $\omega$-Erd\H{o}s cardinals in $L$ by Theorem \ref{thm:non-refl-weakly-rmkp-implies-omega-erdos-in-L}.
\end{proof}

\begin{rem}\label{rem:previous-work}
 Bagaria, Gitman, and Schindler \cite[Theorem 5.4(2)]{BagGitSchGenericVopenka} proved that theory \ref{item:gvp-pi-1-no-pc-rmk} implies the existence of a proper class of virtual rank-into-rank cardinals and asked whether theory \ref{item:gvp-pi-1-no-pc-rmk} is consistent. Gitman and Hamkins \cite[Theorem 12]{GitHamOrdNotDelta2Mahlo} proved that the consistency strength of theory \ref{item:gvp-ord-not-delta-2-mahlo} (and therefore also of theory \ref{item:gvp-pi-1-no-pc-rmk}) is less than $0^\sharp$. In particular they proved that if $0^\sharp$ exists then theory \ref{item:gvp-ord-not-delta-2-mahlo} holds in a generic extension of $L$ (and therefore also in a generic extension of $L_\alpha$ for every Silver indiscernible $\alpha$) by a definable class forcing. In terms of consistency strength, the existence of even a single $\omega$-Erd\H{o}s cardinal is stronger than the existence of a proper class of virtual rank-into-rank cardinals because if $\eta$ is $\omega$-Erd\H{o}s then $V_\eta$ satisfies ZFC + ``there is a proper class of virtual rank-into-rank cardinals.''
\end{rem}

Various other theories may be interposed between theories \ref{item:gvp-ord-not-delta-2-mahlo} and  \ref{item:gvp-pi-1-no-pc-rmk} in Theorem \ref{thm:gvp-equiconsistency}, such as $\text{gVP}$ (or $\text{gVP}({\bfPi}_1)$) + ``there is no $\Sigma_2$-reflecting cardinal'' (or ``there is no remarkable cardinal.'') Such theories are therefore also equiconsistent with the existence of a proper class of $\omega$-Erd\H{o}s cardinals.

\section{Acknowledgments}

The author thanks Joel Hamkins for suggesting the use of G\"{o}del--Bernays set theory to state Lemmas \ref{lem:omega-erdos-is-limit-of-non-refl-wvAe} and \ref{lem:pc-omega-erdos-implies-gvp} in greater generality.


\begin{thebibliography}{10}

\bibitem{BagCnCardinals}
Joan Bagaria.
\newblock ${C}^{(n)}$-cardinals.
\newblock {\em Archive for Mathematical Logic}, 51(3-4):213--240, 2012.

\bibitem{BagGitSchGenericVopenka}
Joan Bagaria, Victoria Gitman, and Ralf Schindler.
\newblock Generic {V}op{\v{e}}nka's {P}rinciple, remarkable cardinals, and the
  weak {P}roper {F}orcing {A}xiom.
\newblock {\em Archive for Mathematical Logic}, 56(1-2):1--20, 2017.

\bibitem{BauIneffabilityII}
James~E. Baumgartner.
\newblock Ineffability properties of cardinals {II}.
\newblock In {\em Logic, foundations of mathematics, and computability theory},
  pages 87--106. Springer, 1977.

\bibitem{GitRamsey}
Victoria Gitman.
\newblock Ramsey-like cardinals.
\newblock {\em Journal of Symbolic Logic}, 76(2):519--540, 2011.

\bibitem{GitHamOrdNotDelta2Mahlo}
Victoria Gitman and Joel~David Hamkins.
\newblock A model of the generic {V}op\v{e}nka principle in which the ordinals
  are not {$\Delta_2$}-{M}ahlo.
\newblock {\em arXiv preprint arXiv:1706.00843}, 2017.

\bibitem{GitSchVirtualLargeCardinals}
Victoria Gitman and Ralf Schindler.
\newblock Virtual large cardinals.
\newblock preprint.

\bibitem{GitWelRamseyII}
Victoria Gitman and P.D. Welch.
\newblock Ramsey-like cardinals {II}.
\newblock {\em Journal of Symbolic Logic}, 76(2):541--560, 2011.

\bibitem{MagRoleSupercompact}
Menachem Magidor.
\newblock On the role of supercompact and extendible cardinals in logic.
\newblock {\em Israel Journal of Mathematics}, 10(2):147--157, 1971.

\bibitem{MagCombinatorialCharacterization}
Menachem Magidor.
\newblock Combinatorial characterization of supercompact cardinals.
\newblock {\em Proceedings of the American Mathematical Society}, 42:279--285,
  1974.

\bibitem{SchProperRemarkable}
Ralf-Dieter Schindler.
\newblock Proper forcing and remarkable cardinals.
\newblock {\em Bulletin of Symbolic Logic}, 6(02):176--184, 2000.

\bibitem{SchProperForcingRemarkableII}
Ralf-Dieter Schindler.
\newblock Proper forcing and remarkable cardinals {II}.
\newblock {\em Journal of Symbolic Logic}, 66(3):1481--1492, 2001.

\bibitem{SchKappaLike}
James~H. Schmerl.
\newblock On $\kappa$-like structures which embed stationary and closed
  unbounded subsets.
\newblock {\em Annals of Mathematical Logic}, 10(3):289--314, 1976.

\bibitem{SilLargeCardinalInL}
Jack~H. Silver.
\newblock A large cardinal in the constructible universe.
\newblock {\em Fundamenta Mathematicae}, 69:93--100, 1970.

\bibitem{SolReiKanStrongAxioms}
Robert~M. Solovay, William~N. Reinhardt, and Akihiro Kanamori.
\newblock Strong axioms of infinity and elementary embeddings.
\newblock {\em Annals of Mathematical Logic}, 13(1):73--116, 1978.

\end{thebibliography}

\end{document}